\documentclass[11pt]{amsart}
\usepackage{
amssymb,
amsmath}

\usepackage[bookmarks,colorlinks,pagebackref]{hyperref} 

\usepackage{graphicx}
\usepackage{hyperref}

\usepackage[latin1]{inputenc}

\usepackage{mathpazo}
\usepackage[scaled=.95]{helvet}
\usepackage{courier}

\usepackage[all,cmtip,2cell]{xy}

\numberwithin{equation}{section}
\theoremstyle{plain} 
\newtheorem{proposition}{Proposition}[section]  
\newtheorem{lemma}[proposition]{Lemma}
\newtheorem{corollary}[proposition]{Corollary} 
\newtheorem{theorem}[proposition]{Theorem} 

\theoremstyle{definition} 

\newtheorem{remark}[proposition]{Remark} 

\newtheorem{example}[proposition]{Example} 
\newtheorem{recalls}[proposition]{Recalls}

\newtheorem{problem}[proposition]{Problem}

\newtheorem{notation}[proposition]{Notation}

\newtheorem{notation and recalls}[proposition]{Notations and Recalls}

\newcommand\Supp{\operatorname{Supp}}
\newcommand\Ass{\operatorname{Ass}}

\newcommand\Ann{\operatorname{Ann}}

\newcommand\Tor{\operatorname{Tor}}
\newcommand\Hom{\operatorname{Hom}}

\newcommand\Ext{\operatorname{Ext}}

\newcommand\Coker{\operatorname{Coker}}

\newcommand\height{\operatorname{height}}
\newcommand\Spec{\operatorname{Spec}}

\author[P.~Schenzel]{Peter Schenzel}
\title[Matlis reflexivity]
{Notes on Matlis reflexive modules}
	
\date{September 1, 2026}

\begin{document}

\begin{abstract} 
	Let $M$ denote a module over a local Noetherian ring $(R,\mathfrak{m})$ with residue field $\Bbbk$ and 
	its injective hull $E_R$. Let $D_R(\cdot) = \Hom_R(\cdot,E_R)$ denote the Matlis duality functor. Then $M$ 
	is called Matlis reflexive whenever the injection $M \to D_R^2(M)$ is an isomorphism.  The following results 
	are proved: (1) For an ideal 
	of $R$ a reflexive module is quasi-complete, generalizing part of Belshoff's results (see \cite{Br}). (2) We improve Z\"oschinger's numerical criterion for Matlis reflexivity (see \cite{Zh2}). (3) As shown by Enochs 
	(see \cite{Ee}) for a prime ideal and the injective hull $E_R(R/\mathfrak{p})$ it follows that $D_R(E_R(R/\mathfrak{p}))$ is the completion of a free $R_{\mathfrak{p}}$-module of rank $\tau_{\mathfrak{p}}$, the Enochs number of $\mathfrak{p}$.  We characterize  when $\tau_{\mathfrak{p}} = 1$, extending author's 
	result (see \cite{Sp16}).
\end{abstract}

\subjclass[2020]
{Primary: 13C05 ; Secondary: 13C11, 13B35}
\keywords{Matlis reflexive module, injective hull, completion}

\maketitle


\section{Introduction}
Let $(R,\mathfrak{m})$ denote a local Noetherian ring with $E_R = E_R(\Bbbk)$ the injective hull of
the residue field $\Bbbk = R/\mathfrak{m}$. Since the work of Matlis (see \cite{Me}) the functor $D_R(\cdot) = \Hom_R(\cdot,E_R)$ 
plays a central role in Commutative Algebra. If $R$ is complete in its $\mathfrak{m}$-adic 
topology $D_R$ provides an anti-equivalence between finitely generated and Artinian $R$-modules 
(see \cite{Me} or \cite{EJ}). Moreover, the natural map $M \to D_R^2(M)$ is pure injective. In case 
$R$ is complete it is an isomorphism, whenever $M$ is finitely generated or Artinian. Furthermore, 
the $R$-module $M$ is called Matlis reflexive if it is an isomorphism. For a complete local ring it 
is shown (see \cite{Ee} or \cite{Zt}) that $M$ is Matlis reflexive if and only if there is a finitely generated 
submodule $N \subseteq M$ such that $M/N$ is Artinian. 
Here we contribute with several new results related to Matlis reflexivity 
for local rings with a view towards new applications in recent research.

\begin{theorem} \label{int-1}
	Let $I$ denote an ideal of $(R,\mathfrak{m})$ a local ring. Suppose that $M$ is Matlis reflexive. 
	Then $M$ is $I$-quasi-complete, i.e. the natural map $M \to \hat{M}^I$ to the $I$-adic completion 
	is onto. Moreover, $\hat{M}^I$ is Matlis reflexive as an $\hat{R}^I$-module  as well as an $R$-module. 
\end{theorem}

For an $R$-module $M$ and $\mathfrak{p} \in \Spec R$ let $\mu(\mathfrak{p},M) = \dim_{k(\mathfrak{p})} 
\Hom_{R_{\mathfrak{p}}}(k(\mathfrak{p}),M_{\mathfrak{p}})$ denote the Bass number of $M$ with respect to 
$\mathfrak{p}$. 

\begin{theorem} \label{in-2}
	Let $M$ denote a module over a local ring $(R,\mathfrak{m})$. Let $\mathfrak{p}$ denote a prime 
	ideal such that $\mu(\mathfrak{p},M) = \mu(\mathfrak{p},D_R^2(M)) \not= 0$. Then these numbers are finite 
	and $R/\mathfrak{p}$ is complete. Moreover $M$ is Matlis reflexive if and only if 
	$\mu(\mathfrak{p},M) = \mu(\mathfrak{p},D_R^2(M))$ for all prime ideals  $\mathfrak{p} \in \Ass_RM$. 
\end{theorem}

The previous numerical criterion of Matlis reflexivity is an improvement of Z\"oschinger's 
result (see \cite{Zh1}) and his arguments. Another feature of  Matlis duality is the study of the dual 
of an injective module, in particular $D_R(E_R(R/\mathfrak{p}))$ for the injective hull $E_R(R/\mathfrak{p})$ 
of $R/\mathfrak{p}$ of prime ideals $\mathfrak{p}$. This was initiated by Enochs (see \cite{Ee} and \cite{EJ}). In particular, 
he has shown 
$$
\Hom_R(E_R(R/\mathfrak{p}), E_R) \cong \widehat{R_{\mathfrak{p}}^{(X_{\mathfrak{p}})}},
$$ 
the completion of a free $R_{\mathfrak{p}}$-module of cardinality $\tau_{\mathfrak{p}} := 
\operatorname{card} X_{\mathfrak{p}}$, the Enochs number of $\mathfrak{p} \in \Spec R$. 
In the case of dimension one a discussion of possible values of $\tau_{\mathfrak{p}}$ is given 
in \cite{Sp15} (see also \ref{mat-16}). Of a particular interest is the case $\tau_{\mathfrak{p}} =1$ 
 investigated in \cite{Sp16}. Note that $\tau_{\mathfrak{m}} = 1$ is always true. For an $R$-module 
 $M$ and a set $X$ recall that $M^{(X)}$ denotes the direct sum of copies of $M$ indexed by $X$ while 
 $M^X$ is the direct product of copies of $M$ indexed by $X$. 

\begin{theorem} \label{in-3}
		Let $\mathfrak{p}\not= \mathfrak{m}$ denote a prime ideal of a local Noetherian ring $(R,\mathfrak{m})$.
	Then $\tau_{\mathfrak{p}} =1$ if and only if $\dim R/\mathfrak{p} = 1$ and $R/\mathfrak{p}$ 
	is complete. 
\end{theorem}

The previous result provides a certain converse to the main statement of \cite{Sp16}.
 In an example we discuss some properties of the injective 
$R$-module $E_R[|T|]$ following \cite[Remark 5.3]{WFT}.
 As a basic reference 
we use Matsumura's textbook \cite{Mh}.

\section{On Matlis Reflexivity and Completion}

Let $(R,\mathfrak{m})$ denote a local Noetherian ring with $E =E_R = E_R(\Bbbk), \Bbbk = R/\mathfrak{m}$, 
the injective hull of the residue field. We denote by $D(\cdot) = D_R(\cdot) = \Hom_R(\cdot,E)$ 
the Matlis duality functor. In the following let $M$ be an arbitrary $R$-module. Let $I \subset R$ 
denote a proper ideal. Before we shall investigate 
Matlis reflexivity we need a result about completion. For an $R$-module $M$ let $X_I(M) = \cap_{n \geq 1} 
I^n M = \varprojlim I^n M$, where the inverse system is given by the inclusions. 
We denote by $\hat{\cdot}^I = \varprojlim(R/I^n \otimes_R \cdot)$ the functor of completion and put 
$X(M) = X_{\mathfrak{m}}(M).$ We recall a well-known fact.

\begin{lemma}\label{mat-4}
	Let $(R,\mathfrak{m})$ denote a local Noetherian ring. 
	Because  the homomorphism $M \to \hat{R}^I \otimes_R M$ is injective for any $R$-module $M$
	it follows that $\hat{R}^I/R$ is $R$-flat
\end{lemma}


Related to the previous result there is the following problem. 

\begin{problem} \label{mat-5}
	Let $(R,\mathfrak{m})$ denote a local Noetherian ring. Since $\hat{R}^I$  is $R$-flat the Matlis dual 
	$\Hom_R(\hat{R}^I,E )$ is an injective $R$-module. By applying the duality to the 
	short exact sequence $0 \to R \to \hat{R}^I \to \hat{R}^I/R \to 0$ yields a splitting 
	\[
	\Hom_R(\hat{R}^I,E) \cong E \oplus \Hom_R(\hat{R}^I/R, E)
	\]
	of injective $R$-modules. Because of $\hat{R}^I/\mathfrak{m} \hat{R}^I \cong \Bbbk$ and adjointness 
	we have  $$\dim_{\Bbbk} \Hom(\Bbbk,\Hom_R(\hat{R}^I,E))= \dim_{\Bbbk} \Hom(\Bbbk,E) = 1.$$
	What are the Bass numbers of $\Hom_R(\hat{R}^I,E)$ for non-maximal prime ideals? Moreover, note that 
	$\Hom_R(\hat{R}^I,E) \cong E$ if and only if $R$ is $I$-adic complete. 
\end{problem}

\begin{recalls} \label{mat-1}
	(A) The natural $R$-homomorphism $\phi_R(M): M \to D_R^2(M)$ is pure injective. The $R$-module 
	$M$ is called Matlis reflexive if $M \to D_R^2(M)$ is an isomorphism. \\
	(B) Let $N \subset M$ be a submodule. It is easy to see that $M$ is Matlis reflexive 
	if and only if $N$ and $M/N$ are Matlis reflexive. \\
	(C) Suppose that $(R,\mathfrak{m})$ is complete. Then it follows (see \cite{Ee} or \cite{Zt}) 
	that $M$ is Matlis reflexive if and only if there is a finitely generated submodule $N \subset M$ 
	such that $M/N$ is Artinian. Let $(R,\mathfrak{m})$  denote an arbitrary local Noetherian ring and let $M$ 
	be a Matlis reflexive module. By modifying the original argument  it is shown (see \cite[Corollary 7]{BEe}) 
	that $M$ admits a finitely generated $R$-module $N$ such that $M/N$ is an Artinian $R$-module. \\
	(D) For $M$ a finitely generated $R$-module we have 
	\[
	M \to D^2_R(M) = \Hom_R(\Hom_R(M,E),E) \cong M \otimes_R \Hom_R(E,E) \cong 
	M_R \otimes_R \hat{R} \cong \hat{M}, 
	\]
	so that $M$ is Matlis reflexive if and only if it is complete. 
\end{recalls}

As an application of the  Recall \ref{mat-1} (B) we prove  the following elementary observations.  

\begin{theorem} \label{lem-1}
	Let $(R,\mathfrak{m})$ denote a local ring. Let $M$ be an $R$-module. 
	\begin{itemize}
		\item[(a)] If $M$ is Matlis reflexive, then $R/\mathfrak{p}$ is complete for any $\mathfrak{p} \in \Ass_RM$.
		\item[(b)] Let $M$ be finitely generated and $I = \Ann_RM$. Then $R/I$ is complete if and only if $M$ 
		is Matlis reflexive.
		 \item [(c)] Let $M$ be finitely generated. Then $M$ is Matlis reflexive if and only if 
			$R/\mathfrak{p}$ is complete for all minimal $\mathfrak{p} \in \Ass_R M$.
	\end{itemize}
\end{theorem}

\begin{proof}
	Let $\mathfrak{p} \in \Ass_RM$, then there is an injection $0 \to R/\mathfrak{p} \to M$. Then (a) follows by 
	\ref{mat-1} (B) because $R/\mathfrak{p}$ is Matlis reflexive and therefore complete. 
	 For the case of (b) let $M$ be finitely generated with $m_1,\ldots,m_k$ a basis of $M$. Then the  map 
	$R/I \to M^k$ defined by $r +I \mapsto (rm_i)_{i=1}^k$  is injective as easily seen. If $M$ is Matlis reflexive, 
	then $R/I$ is Matlis reflexive as a submodule of $M^k$ and therefore complete. Since $M$ is generated by 
	$k$ elements there is an epimorphism $(R/I)^k \to M \to 0$. If $R/I$ is complete it is Matlis reflexive 
	and so is $M$.  For the proof of (c) the "only if part" follows by (a). Conversely suppose that $R/\mathfrak{p}$ is complete for all minimal $\mathfrak{p} \in \Ass_R M$. Therefore $R/\mathfrak{q}$ is Matlis reflexive for all $\mathfrak{q} \supset \mathfrak{p}$ because of the epimorphism $R/\mathfrak{p} \to R/\mathfrak{q} \to 0$ 
	and a minimal $\mathfrak{p} \in \Ass_R M$.. 
		
	By \cite[IV, \& 4]{Bn1} 
	there is a finite sequence $0 = M_n \subset M_{n-1}  \subset \ldots \subset M_0 = M$ of submodules of $M$ 
	such that $M_i/M_{i+1} \cong R/\mathfrak{p}_i, \mathfrak{p}_i \in \Spec R,$ for $i = 0, \ldots, n-1,$
	with the property 
	\[
	\Ass_RM \subseteq \{\mathfrak{p}_0, \ldots, \mathfrak{p}_{n-1}\} \subseteq \Supp_R M
	\]
	and the minimal elements  of these three sets are the same. That is, any of these $\mathfrak{p}_i, i = 0,\ldots,n-1,$ 
	contains a minimal prime of $\Ass_RM$ and is therefore Matlis reflexive. Because of the short exact sequence 
	\[
	0 \to M_{i+1}\to M_i \to R/\mathfrak{p}_i \to 0, \; i = 0, \ldots, n-1,
	\]
	and descending induction it follows that $M$ is Matlis reflexive.
\end{proof}

The result in \ref{lem-1} (b) is a different proof of \cite[Theorem 9]{BEe} in the case of a local ring. The statement 
does not hold for arbitrary modules. 
In the following we shall investigate Matlis reflexive modules over a non-complete local ring. 
For an $R$-module $M$ let $\hat{M}^I = \varprojlim M/I^n M$ its $I$-adic completion. 

\begin{theorem} \label{mat-2}
	Let $M$ be a Matlis reflexive $R$-module and let $X_I = X_I(M) = \cap_{n \geq 1} I^n M$. 
	Then there is a short exact sequence $0 \to X_I \to M \to \hat{M}^I \to 0$. That is, $M$ is 
	$I$-quasi-complete, moreover $\hat{M}^I$ and $X_I(M)$ are Matlis reflexive over $R$. 
\end{theorem}

\begin{proof}
	By the assumption (see \ref{mat-1} (B)) $I^nM$ as well as $M/I^nM$ 
	are Matlis reflexive for all $n \geq 1$. With the inclusion maps $\{I^nM \}_{n \geq 1}$ forms 
	an inverse system with $X_I(M) = \varprojlim I^nM$. By applying $D_R$ it follows that 
	$\{D_R(I^n M)\}_{n \geq 1}$ forms a direct system with a short exact sequence 
	\[
	0 \to \oplus_{n \geq 1} D(I^nM) \stackrel{\Phi}{\longrightarrow} 
	\oplus_{n \geq 1} D(I^nM) \to \varinjlim D(I^nM) \to 0
	\]
	as follows by the definition of the direct limit. By applying $D(\cdot)$ again it provides a commutative 
	diagram with exact rows 
	\[
		\xymatrix{
		0 \ar[r] & \varprojlim I^nM \ar[r]  \ar[d] & \prod_{n \geq 1} I^nM  \ar[r] \ar[d] & \prod_{n \geq 1} I^nM 
		\ar[r] \ar[d] & \varprojlim{}^1 I^nM  \ar[r] & 0 \\
		0 \ar[r] &  \varprojlim D_R^2(I^nM) \ar[r]  & \prod_{n \geq 1} D_R^2(I^nM) \ar[r]  & \prod_{n \geq 1} D_R^2(I^nM) 
		\ar[r] & 0 & 
	}
	\]
	because of $D_R( \varinjlim D_R(I^nM) ) \cong \varprojlim D_R^2(I^nM)$. Since the 
	vertical maps are isomorphisms it yields the vanishing $\varprojlim{}^1 I^nM = 0$. By the definitions, 
	the inverse system of short exact sequences $0 \to I^nM \to M \to M/I^nM \to 0$ 
	for all $n \geq 1$ implies the exact sequence 
	\[
	0 \to  \varprojlim I^nM  \to M \to  \varprojlim M/I^nM \to  
	\varprojlim{}^1 I^nM \to 0.
	\]
	By the vanishing it follows that $M \to \hat{M}^I$ is onto and therefore $M$ is $I$-quasi-complete. 
	The rest is clear by \ref{mat-1} (B).
\end{proof}

In the following we shall apply the previous results in order to derive some further 
properties of Matlis reflexive modules. In the following we shall prove a slight extension of Belshoff's
results (see \cite[Theorem 2]{Br})

\begin{corollary} \label{mat-3}
	Let $M$ denote a Matlis reflexive $R$-module. 
	\begin{itemize}
		\item[(a)] The natural map $\hat{R}^I \otimes_RM \to \hat{M}^I$ is onto. 
		\item[(b)] The natural map $M \to \hat{R}^I \otimes_RM$ is an $R$-isomorphism. 
		\item[(c)] For two ideals $J \subseteq I$ the natural map $\hat{M}^J \to \hat{M}^I$ is onto.
	\end{itemize}
\end{corollary}

\begin{proof}
	The short exact sequence $0 \to R \to \hat{R}^I \to \hat{R}^I/R \to 0$ and the results 
	of \ref{mat-4} provide a commutative diagram with exact rows
	\[
		\xymatrix{
		& 0 \ar[r]  & M \ar[r] \ar@2{-}[d] &\hat{R}^I\otimes_R M 
		\ar[r] \ar[d]^{\alpha} &\hat{R}^I/R \otimes_RM\ar[r] & 0 \\
		0 \ar[r] &  X_I \ar[r]  & M \ar[r]  & \hat{M}^I
		\ar[r] & 0. & 
	}
	\]
	Then $\alpha: \hat{R}^I \otimes_RM \to \hat{M}^I$ is onto as is easily seen, i.e. (a) holds. 
	In the case $M$ is finitely generated $\alpha$ is an isomorphism and $X_I = 0$ by the Krull 
	Intersection Theorem. By the diagram (b) is true. For the general case 
	$M$ is the direct limit $M = \varinjlim_{\lambda \in \Lambda} M_{\lambda}$ of finitely generated 
	Matlis reflexive submodules $M_{\lambda} \subset M$. Then (b) follows by passing to the direct limit. 
	Since $X_J(M) \subseteq X_I(M)$ the statement in (c) is a consequence of \ref{mat-2}.
\end{proof}

In case $M$ is a finitely generated $R$-module and $I = \mathfrak{m}$ then $M$ is Matlis reflexive 
if and if $M \to \hat{R} \otimes_R M$ is an isomorphism (see  \cite[Theorem 3]{BEe}). In the following we shall discuss 
the property of Matlis duality and completions a bit more. 

\begin{corollary} \label{mat-6}
	Let $(R,\mathfrak{m})$ denote a local Noetherian ring. Let $M$ be a Matlis 
	reflexive $R$-module. Then $\hat{M}^I$ is Matlis reflexive as an $R$-module as well as an $\hat{R}^I$-module. 
\end{corollary}

\begin{proof}
	The first claim holds by \ref{mat-2}. 
	By view of \ref{mat-1} (C) there is a finitely generated submodule $N \subseteq M$ such that $M/N$ is Artinian.
	Let $N'$ denote the image of $N$ in $\hat{M}^I \cong M/X_I(M)$. Then $N/X_I(M) \cap N \cong N'$ is a finitely generated $\hat{R}^I$-sub module of $\hat{M}^I = M/X_I(M)$.  Moreover $M/(N,X_I(M))$ 
	is as an image of $M/N$ an Artinian $R$-module. Therefore there is an embedding of 
	$\hat{M}^I/N'$ into a finite direct sum of copies of $E$. Then we have $E = E_{\hat{R}^I}$, see \cite[tag 08Z4]{stacks}. 
	Therefore 
	as an $\hat{R}^I$-module $\hat{M}/N'$ is an Artinian module. 
	So the claim follows by \ref{mat-2} (B). 
\end{proof}

Further results on Matlis reflexivity of completions are summerized in the following.

\begin{corollary} \label{mat-7}
	Let $M$ denote a module over a local Noetherian ring $(R,\mathfrak{m})$. We investigate  the following 
	conditions for ideals $I \subset R$ and an $R$-module $M$:
	\begin{itemize}
		\item[(i)] $M$ is Matlis reflexive.
		\item[(ii)] $\hat{M}^I$ is Matlis reflexive for any ideal $I$.
		\item[(iii)] $\hat{M}^I$ is Matlis reflexive for some ideal $I$.
		\item[(iv)] $M/X_I(M)$ is Matlis reflexive.
	\end{itemize}
	Then there are the folowing results:
	\begin{itemize}
		\item[(a)] We have the implications (i) $\Longrightarrow$ (ii) $\Longrightarrow$ (iii) $\Longrightarrow$(iv).
		\item[(b)] If $M$ is $I$-separated, that is 
		$X_I(M) = 0$, then  $M \cong \hat{M}^I$. 
		\item[(c)] If $M$ is finitely generated and $\hat{M}^I$ is reflexive for some ideal $I$, then $M$ is reflexive. 
	\end{itemize}
\end{corollary}

\begin{proof}
	The proofs follows by \ref{mat-2}. For (c) recall the Krull Intersection Theorem. 
\end{proof}

The previous results in   \ref{mat-7}  are related to Belshoff's argument (see \cite{Br}).  It generalizes part of
\cite[Theorem 3]{BEe} to the case of not necessarily finitely generated modules. We continue with a diagram 
relating the double dual of $M$ and $\hat{M}^I$, related to the diagram in the proof of \cite[Theorem 3]{BEe}.

\begin{proposition} \label{mat-8}
		Let $M$ denote a module over a local Noetherian ring $(R,\mathfrak{m})$. For an ideal $I \subset R$ 
		we have the following commutative diagram with exact rows
		 \[
		\xymatrix{
			0 \ar[r] &M  \ar[r]   \ar[d] & \Hom_R(\Hom_R(M,E),E) \ar^{\alpha}[d]  \\
			0 \ar[r] &  \hat{M}^I \ar[r]  &  \Hom_{\hat{R}^I}(\Hom_{\hat{R}^I}(\hat{M}^I , E_{\hat{R}^I}), E_{\hat{R}^I}). 
		}
		\]
\end{proposition}

\begin{proof}
	Because of $E_{\hat{R}^I} \cong E_R$ (see e.g.  \cite[tag 08Z4]{stacks}) it follows that 
	$\Hom_{\hat{R}^I}(\hat{R}^I \otimes_R M, E_{\hat{R}^I})\cong \Hom_R(M,E)$ by adjointness. 
	The natural map $ \hat{R}^I \otimes_RM \to \hat{M}^I$ induces a map 
	\[
	\Hom_{\hat{R}^I}(\hat{M}^I, E_{\hat{R}^I}) \to  \Hom_{\hat{R}^I}(\hat{R}^I \otimes_R M, E_{\hat{R}^I} )
	\cong  \Hom_R(M,E).
	\]
	Tensoring by $\hat{R}^I$ it provides a natural homomorphism 
	\[
		\Hom_{\hat{R}^I}(\hat{M}^I, E_{\hat{R}^I}) \to
	\Hom_{\hat{R}^I}(\hat{M}^I, E_{\hat{R}^I})  \otimes_R \hat{R}^I \to  \Hom_R(M,E) \otimes_R \hat{R}^I.
	\]
	By applying the duality functor $D_{\hat{R}^i}(\cdot)$ yields a map 
	\[
		\Hom_{\hat{R}^I}( \Hom_R(M,E) \otimes_R \hat{R}^I,  E_{\hat{R}^I}) \to 
	\Hom_{\hat{R}^I}(\Hom_{\hat{R}^I}(\hat{M}^I, E_{\hat{R}^I}),  E_{\hat{R}^I}) .
	\]
	By adjointness the first of these modules is isomorphic to 
	\[
	\Hom_R(\Hom_R(M,E),\Hom_{\hat{R}^I}(\hat{R}^I, E_{\hat{R}^I})) \cong 
	\Hom_R(\Hom_R(M,E),E)).
	\]
	This fits into the  commutative diagram above. 
\end{proof}

\section{On Matlis Reflexivity and Base Change}
\begin{remark} \label{mat-9}
	Let $V = \Bbbk^{(\Lambda)}$ denote a vector space over a field $\Bbbk$. Then the 
	natural injective map $V \to \Hom_{\Bbbk}( \Hom_{\Bbbk}(V, \Bbbk), \Bbbk)$ is an isomorphism 
	if and only if $\Lambda$ is finite. This is generalized in the following way: 
	No infinite direct sum of nonzero $R$-modules is Matlis reflexive (see \cite[Lemma 6]{BEe}). 
	See also \cite{Kh}, where the result is used in order to show 
	that Matlis reflexive modules form a Krull-Schmidt category.
	Moreover, a free $R$-module is Matlis reflexive if and only if it is of finite rank and $R$ is complete. 
\end{remark}

\begin{proposition} \label{mat-15}
	Let $P$ denote a projective module over a local ring $(R,\mathfrak{m})$. Suppose that 
	$P$ is Matlis reflexive. Then  $R$ is complete and 
	$P$ is a free $R$-module of finite rank. 
\end{proposition}

\begin{proof}
	First note that $P$ is a direct summand of a free $R$-module $F$ , so that the natural map $P \to F$ is split injective. Because $D_R^2(\cdot)$ commutes 
	with finite direct sums the map $D_R^2(P) \to D_R^2(F)$ is a split injection. Then the commutative 
	diagram 
	\[
	\xymatrix{
		P \ar[r]^-{\cong}  \ar[d] & D_R^2(P)\ar[d]\\
		F  \ar[r]  & D_R^2(F)
	}
	\]
	implies $F\cong D_R^2(F)$ since the vertical maps are split injective. 
	Let $F = R^{(\Lambda)}$ for a certain set $\Lambda$.  Then $\Lambda$ is finite 
	as follows by \cite[Lemma 6]{BEe}. That is $R^{(\Lambda)} = \hat{R}^{(\Lambda)}$ 
	for the finite $\Lambda$ and $R = \hat{R}$. Then the rank of 
	$P$ is also finite. But a finitely generated projective module over a local ring is free. 
\end{proof}

Next we are interested in a certain change of ring result for the Matlis double dual of a module.

\begin{lemma}\label{mat-10}
	Let $(R,\mathfrak{m})$ denote a local Noetherian ring and let $(R,\mathfrak{m}) \to (S,\mathfrak{n})$ 
	be a local homomorphism such that $S$ is a finitely generated $R$-module. Then there are 
	isomorphisms
	\begin{itemize}
		\item[(a)] $D^2_S(\Hom_R(S,M)) \cong \Hom_R(S,D^2_R(M))$ and 
		\item[(b)] $D^2_S(S \otimes_R M) \cong S \otimes_RD^2_R(M)$
	\end{itemize}
for any $R$-module $M$.
\end{lemma}

\begin{proof}
	First of all note that $E_S \cong \Hom_R(S,E_R)$, see e.g. \cite[10.6.1]{SS}. Then 
	\[
	D_S(\Hom_R(S,M)) \cong \Hom_S(\Hom_R(S,M), \Hom_R(S,E_R)) \cong D_R(\Hom_R(S,M))
	\]
	as follows by adjointness. Furthermore, we have the isomorphisms
	\[
	D_S^2(\Hom_R(S,M)) \cong \Hom_S(D_R(\Hom_R(S,M)), \Hom_R(S,E_R)) \cong 
	\Hom_R(S, D^2_R(M))
	\]
	which proves the claim in (a). For the proof of (b) we get 
	\begin{gather*}
	D_S(S \otimes_RM) \cong  \Hom_S(S\otimes_RM,E_S) \cong \Hom_S(S\otimes_RM,\Hom_R(S,E_R))\cong \\ \Hom_R(S\otimes_RM,E_R)
	\cong \Hom_R(S,\Hom_R(M,E_R)).
	\end{gather*}
	Moreover, there are isomorphisms 
	\begin{gather*}
	D_S^2(S\otimes_RM) \cong \Hom_S( \Hom_R(S,\Hom_R(M,E_R)), \Hom_R(S,E_R )) \cong \\
	\Hom_R(\Hom_R(S,\Hom_R(\Hom_R(M,E_R),E_R)).
	\end{gather*}
	Since $S$ is finitely generated as $R$-module and $E_R$ is injective the last module in the above 
	sequence of isomorphisms is isomorphic to $S\otimes_RD_R^2(M)$, so the proof of (b) is complete. 
\end{proof}

Now let as above $\phi_R(M) : M \to D^2_R(M)$ denote the natural injection and $\Coker \phi_R(M)$ the 
cokernel of the embedding. Then we have a slight generalization of one of Z\"oschinger's results 
(see \cite[Lemma 1.4]{Zh1}). 

\begin{corollary} \label{mat-11}
	With the notation of \ref{mat-10} we have the following isomorphisms 
	\begin{itemize}
		\item[(a)] $\Coker \phi_S(\Hom_R(S,M)) \cong \Hom_R(S,\Coker \phi_R(M))$ and 
		\item[(b)]$\Coker \phi_S(S\otimes_RM) \cong S \otimes_R \Coker \phi_R(M)$
	\end{itemize}
	for any $R$-module $M$.
\end{corollary}

\begin{proof}
	For the proof use the short exact sequence $0 \to M \to D_R^2(M) \to \Coker \phi_R(M) \to 0$, recall that 
	$  M \to D_R^2(M)$ is pure injective and apply $\Hom_R(S,\cdot)$. 
	Then there is a commutative diagram with exact rows
	\[
	\xymatrix{
	0 \ar[r] & \Hom_R(S,M) \ar[r]  \ar[d] &  \Hom_R(S,D_R^2(M))  \ar[r] \ar[d] & 
	 \Hom_R(S,\Coker \phi_R (M)) 
	 \ar[d]  \ar[r] &0 \\
	0 \ar[r] &  \Hom_R(S,M) \ar[r]  &D^2_S( \Hom_R(S,M) ) \ar[r]  & 
	\Coker \phi_S( \Hom_R(S,M) )\ar[r] & 0 
	}
	\]
	The two vertical maps at the left are isomorphisms (see \ref{mat-10} for the second). Then
	the first result in (a) follows. 
	The second statement is proved similarly by applying $S\otimes_R \cdot$ to the above short exact 
	sequence. 
\end{proof}

In the following we shall summarize a few notations, needed in the following. For a prime ideal $\mathfrak{p} \in 
\Spec R$ let $k(\mathfrak{p}) = R_{\mathfrak{p}}/\mathfrak{p} R_{\mathfrak{p}}$ the quotient 
field of $R/\mathfrak{p}$. 

\begin{notation} \label{mat-12}
	(A) For an $R$-module $M$ let $\mu(\mathfrak{p},M) = \dim_{k(\mathfrak{p})} \Hom_{R_{\mathfrak{p}}}(k(\mathfrak{p}),M_{\mathfrak{p}})$. 
	Note that $\mu(\mathfrak{p},M) = 
	\dim_{k(\mathfrak{p})} (\Hom_R(R/\mathfrak{p},M))_{\mathfrak{p}}$. If $E(M)$ denotes the injective hull 
	of $M$, then $E(M) \cong \oplus_{\mathfrak{p} \in \Spec R} E_R(R/\mathfrak{p})^{(X(\mathfrak{p},M))}$ 
	with $\operatorname{card} X(\mathfrak{p},M) = \mu(\mathfrak{p},M)$ as follows by Matlis structure 
	theory (see e.g. \cite{EJ} for the details).\\
	(B) For a prime ideal $\mathfrak{p}$ it follows that $k(\mathfrak{p}) \otimes_{R_{\mathfrak{p}}} M_{\mathfrak{p}} 
	\cong (R/\mathfrak{p} \otimes_R M)_{\mathfrak{p}} \cong k(\mathfrak{p})^{(Y(\mathfrak{p},M))}$. We define 
	$\rho(\mathfrak{p},M) = \operatorname{card} Y(\mathfrak{p},M)$, the rank of $M_{\mathfrak{p}}$. In case $R$ 
	is a domain, $\rho(0,M) = \dim_Q M \otimes_R Q$, where $Q$ denotes the quotient field of $R$.
\end{notation}

\begin{theorem} \label{mat-13}
	Let $(R,\mathfrak{m})$ denote a Noetherian local ring. Let $M$ be an $R$-module.  Let $\mathfrak{p}$ denote 
	a prime ideal of $R$. 
	\begin{itemize}
		\item[(a)] If $\mu(\mathfrak{p},M) = \mu( \mathfrak{p},D_R^2(M)) \not= 0$, then these numbers are finite 
		and $R/\mathfrak{p}$ is complete.
		\item[(b)] If $\rho(\mathfrak{p},M) = \rho( \mathfrak{p},D_R^2(M)) \not= 0$,  then these numbers are finite 
		and $R/\mathfrak{p}$ is complete.
	\end{itemize}
\end{theorem}

\begin{proof} (a): 
	First of all we show that if $\mu(\mathfrak{p},M) = \mu(\mathfrak{p},D_R^2(M))$ then these numbers are finite.
	To this end we first assume that $R$ is a domain and $\mathfrak{p} = 0$. Then $\mu(0,X) = \dim_Q X\otimes_RQ$, 
	the rank of $X$, where $Q$ denotes the quotient field of $R$. Now we follow an argument 
	by Z\"oschinger (see \cite{Zh2}). Let $F = R^{(\Lambda)} \subseteq M$ denote a free submodule of $M$ 
	with $\operatorname{card} \Lambda = \operatorname{rank} M$. Then $M/F$ is a torsion module and 
	$F\otimes_RQ \cong M\otimes_RQ$. The diagram of injections 
	\[
		\xymatrix{
	     F \ar@{>->}[r]  \ar@{>->}[d] & M \ar@{>->}[d]\\
		D_R^2(F)  \ar@{>->}[r]  & D_R^2(M)
	}
	\]
	provides that $F \otimes_RQ \cong D_R^2(F) \otimes_RQ$ and $\operatorname{rank} F = 
	\operatorname{rank}  D_R^2(F)$.  
	In order to continue we have to show that $\Lambda$ 
	is finite. Since $F = R^{(\Lambda)}$ it induces an embedding $E_R^{(\Lambda)}\rightarrowtail 
	E_R^{\Lambda} = D_R(F)$ 
	that is a split injection because $E_R^{\Lambda}$ is an injective $R$-module. 
	By applying $D_R(\cdot)$ it induces 
	a split injection $ \hat{R}^{\Lambda} \rightarrowtail D_R^2(F)$. Because of the injections 
	$F = R^{(\Lambda )}  \rightarrowtail R^{\Lambda} \rightarrowtail \hat{R}^{\Lambda} \rightarrowtail D^2_R(F)$ 
	it provides 
	\[
	R^{(\Lambda)} \otimes_R Q \cong  \hat{R}^{\Lambda} \otimes_R Q  \;  \mbox{ and } \; 
	\operatorname{rank} R^{(\Lambda)} = \operatorname{rank} R^{\Lambda}
	\]
	recall that $\operatorname{rank} F = \operatorname{rank} D_R^2(F)$. 
	Now it follows that 
	$R/\mathfrak{m} \otimes_RR^{\Lambda} \cong (R/\mathfrak{m})^{\Lambda}$ (see \cite[Chapter 1, § 2, Ex. 9]{Bn1}) 
	and $R/\mathfrak{m} \otimes_RR^{(\Lambda)} \cong (R/\mathfrak{m})^{(\Lambda)}$, that is 
	$\dim \Bbbk^{(\Lambda)} = \dim \Bbbk^{\Lambda}  = \dim \Hom_{\Bbbk}(\Bbbk^{(\Lambda)} , \Bbbk)$. But 
	this implies that $\Lambda$ is finite (see \ref{mat-4}). Since $\Lambda$ is finite by the isomorphism 
	above it follows that $ (\hat{R}^{(\Lambda)}/ R^{(\Lambda)})  \otimes_R Q \cong ((\hat{R}/R)  \otimes_R Q)^{(\Lambda)} = 0$
	and therefore $(\hat{R}/R) \otimes_R Q = 0$.  Now we use the short exact sequence 
	$0 \to R \to Q \to Q/R \to 0$ and tensor it by $\hat{R}/R$. Because $Q$ is $R$-flat and because of $\Tor_1^R(\hat{R}/R, Q/R) = 0$
	(see \ref{mat-4}) it induces an injection $0 \to \hat{R}/R \to (\hat{R}/R) \otimes_R Q$, which shows that $R = \hat{R}$. 
	
	In order to show the general case first note $\mu(\mathfrak{p},M) = \mu(\mathfrak{p}, \Hom_R(R/\mathfrak{p},M))$. 
	Because of \ref{mat-10} (A) we have $D_{R/\mathfrak{p}}^2(\Hom_R(R/\mathfrak{p},M)) \cong \Hom_R(R/\mathfrak{p},D_R^2(M))$. It follows that 
	\[
	\mu(\mathfrak{p}, D_R^2(M)) = \mu(\mathfrak{p}, \Hom_R(R/\mathfrak{p},D_R^2(M))) = 
	\mu(\mathfrak{p}, D_{R/\mathfrak{p}}(\Hom_R(R/\mathfrak{p},M))). 
	\]
	Whence we may consider the module $\Hom_R(R/\mathfrak{p},M)$ over the domain  $R/\mathfrak{p}$. 
	Then the conclusion follows as before. \\
	(b): If $R$ is a domain we conclude as in (a) because of $\mu(0,X) = \rho(0,X)$. In the general case 
	we have that $\rho (\mathfrak{p},M) = \rho(0,M/\mathfrak{p} M)$ as easily seen. Moreover 
	by view of \ref{mat-10} (B) it follows that $\rho(\mathfrak{p}, D_R^2(M)) = \rho(0,D_R^2(M)/\mathfrak{p}D_R^2(M)) = 
	\rho(0,D_{R/\mathfrak{p}}^2(M/\mathfrak{p}M))$. That is, 
	we may consider $M/\mathfrak{p}M$ as a module over the domain $R/\mathfrak{p}$. Then the claim follows as 
	before. 
\end{proof}

As an application we get the following simplification of Z\"oschinger's result (see \cite[Satz 1.1]{Zh2}). In his theorem  
the coincidence of the $\mu$-numbers are required for all $\mathfrak{p} \in \Ass_R \Coker \phi_R(M)$. 
In fact $\Ass_R \Coker \phi_R(M)$ is difficult to describe.

\begin{theorem} \label{mat-14}
	Let $(R,\mathfrak{m})$ denote a local ring and let $M$ be an $R$-module. Then $M$ is Matlis 
	reflexive, i.e. $M \cong D_R^2(M)$, if and only if $\mu(\mathfrak{p},M) = \mu( \mathfrak{p},D_R^2(M))$
	for all $\mathfrak{p} \in \Ass_R M$.  
	If this is the case $R/\mathfrak{p}$ is complete. 
\end{theorem}

\begin{proof}
	It will be enough to show the only if part. Assume the contrary. Then $\Coker \phi_R(M) \not= 0$ 
	and there is a prime ideal $\mathfrak{p} \in \Ass_R \Coker_R(M)$. Therefore 
	\[
	0 \not= \Hom_R(R/\mathfrak{p},\Coker \phi_R(M)) \cong
	\Coker \phi_{R/\mathfrak{p}} (\Hom_R(R/\mathfrak{p},M))
	\]
	(see \ref{mat-11}). This implies $\Hom_R(R/\mathfrak{p},M) = 0:_M \mathfrak{p} \not= 0$ and 
	$$
	(\Coker \phi_{R/\mathfrak{p}} (\Hom_R(R/\mathfrak{p},M)))_{\mathfrak{p}} \not=0.
	$$
	Now we pass to the domain $S = R/\mathfrak{p}$ with its quotient field $Q$ and the non-zero $S$-module $N = \Hom_R(R/\mathfrak{p},M)$. By the assumption $0 \not= \dim_Q N \otimes_SQ = \dim_Q D_S^2(N) \otimes_SQ$ and these numbers are finite as follows by virtue of \ref{mat-13}. That is, $\mathfrak{p} \in \Ass_R M$ and the injection of the finite dimensional $Q$-vector spaces $N\otimes_SQ \to 
	D_S^2(N) \otimes_SQ$ is an isomorphism. Therefore 
	$$
		(\Coker \phi_{R/\mathfrak{p}} (\Hom_R(R/\mathfrak{p},M)))_{\mathfrak{p}} \cong (\Coker \phi_S(N)) \otimes_S Q = 0,
	$$
	a contradiction. 
\end{proof}

Dual to the notion of associated prime ideals of a module  Z\"oschinger (see \cite{Zh3}) studied  the notion 
of coassociated prime ideals. We say $\mathfrak{p} \in \Spec R$ is coassociated to $M$ whenever there is a 
submodule $N\subseteq M$ such that $M/N$ is Artinian and $\Ann_R M/N = \mathfrak{p}$. The set of all 
coassociated prime ideals of $M$ is denoted by $\operatorname{Coass}_R M$. In case of a local ring $(R,\mathfrak{m})$ it follows that $\mathfrak{p} \in \operatorname{Coass}_R M$ if and only if $\mathfrak{p} 
\in \Ass_R D_R(M)$. Moreover, $M/\mathfrak{p}M \not= 0$ if and only if $\mathfrak{p} \in \operatorname{Coass}_R M$ because $D_R(M/\mathfrak{p}M )\cong \Hom_R(R/\mathfrak{p},D_R(M))$. It is an open 
question to the author whether 
a result corresponding to \ref{mat-14} holds for $\rho(\mathfrak{p}, \cdot)$ and the coassociated primes.

\section{Injective Hulls}

For several reasons there is some interest in understanding the 
Matlis dual of an injective $R$-module, in particular the dual of 
$E_R(R/\mathfrak{p}), \mathfrak{p} \in \Spec R$, 
the components of any injective $R$-module $I$. 
For an arbitrary index set $X$ and an $R$-module $M$ we use the notations 
$M^{(X)}$ and $M^{X}$ as defined above. 
As above, for an $R$-module 
It was shown by Enochs (see \cite[Theorem 4.3]{Ee} and 
also \cite[3.3.14, 3.4.1]{EJ})
\[
\Hom_R(E_R(R/\mathfrak{p}), E_R) \cong \Hom_R(E_R(R/\mathfrak{p}), E_R(R/\mathfrak{p})^{(X_\mathfrak{p})}) 
\cong \widehat{R_{\mathfrak{p}}^{(X_{\mathfrak{p}})}}.
\]
for a certain set $X(\mathfrak{p})$. Matlis'  structure theory shows that the cardinality 
$\tau_\mathfrak{p} := \operatorname{card}X_{\mathfrak{p}}$ is the dimension of the $k(\mathfrak{p})$-vector space $\Hom_R(k(\mathfrak{p}),E_R)$.
The adjunction formula gives the 
following isomorphisms
\[
\Hom_R(k(\mathfrak{p}), E_R) \cong \Hom_{R/\mathfrak{p}}(k(\mathfrak{p}), \Hom_R(R/\mathfrak{p},  E_R)) \cong 
\Hom_{R/\mathfrak{p}}(k(\mathfrak{p}), E_{R/\mathfrak{p}}).
\]
For the last isomorphism note that $\Hom_R(R/\mathfrak{p}, E_R) \cong E_{R/\mathfrak{p}}$.
Therefore $\Hom_R(E_R(R/\mathfrak{p}), E_R)$ is the completion of a free $R_{\mathfrak{p}}$-module of rank $\tau_{\mathfrak{p}}$
with
 \[
\tau_{\mathfrak{p}} = \dim_{k(\mathfrak{p})} \Hom_R(k(\mathfrak{p}), E_R) = 
\dim_{k(\mathfrak{p})}  \Hom_{R/\mathfrak{p}}(k(\mathfrak{p}), E_{R/\mathfrak{p}}).
\]
That is, we  may pass to the domain $S := R/\mathfrak{p}$ with its quotient field $Q$. That is,
\[
\Hom_S(Q,E_S) \cong \Hom_S(Q,Q^{(X_{\mathfrak{p}})}) \cong Q^{(X_{\mathfrak{p}})}
\]
with $\tau_{\mathfrak{p}} = \operatorname{card}X_{\mathfrak{p}} = \dim_Q \Hom_S(Q,E_S)$. 
Note that $E_S(S) \cong Q$, the quotient field of $S$.
For a prime ideal 
$\mathfrak{p} \in \Spec R$ we call $\tau_{\mathfrak{p}}$ the Enochs number of $\mathfrak{p}$. 
An overview of certain values for $\tau_{\mathfrak{p}}$ is given in the following.

\begin{remark} \label{mat-16}
	(A) Let $\mathfrak{p}$ denote a one dimensional prime 
	ideal of a local ring $(R, \mathfrak{m})$. 
	Suppose that $R/\mathfrak{p}$ is a Gorenstein ring and analytically unramified. 
	All the cardinal numbers $\tau_{\mathfrak{p}}$ that occur for particular examples of rings are exactly 
	\begin{itemize}
		\item[(1)] any infinite cardinal number,
		\item[(2)] among the finite cardinal numbers exactly those of the form $p^t, p$
		being a prime number and $t$ a non-negative integer.
		\item[(3)] $\tau_{\mathfrak{p}} =1$ if and only if $R/\mathfrak{p}$ is complete. 
	\end{itemize} 
	When $\tau_{\mathfrak{p}}$ is finite and not equal to 1 $R/\mathfrak{p}$ is necessarily of prime characteristic. 
	For the details see \cite[Theorem 1.2]{Sp15} and also \cite{Sp16}.\\
	(B) Let $S$ be a one dimensional local  domain with quotient field $Q$. The obstruction for the 
	completeness of $S$ is given by $\hat{S}/S \cong \Ext_S^1(Q,S)$  (see \cite{Sp16}). Then the possible values 
	for $\dim_Q \hat{S}/S$ are shown by C. U. Jensen (see \cite{Jcu}). In his construction C. U. Jensen used  modifications of Nagata's "bad" local rings see \cite{Nm}).
\end{remark}

In the following we shall characterize those prime ideals $\mathfrak{p} \in \Spec R$ such that 
$\tau_{\mathfrak{p}} = 1$. This improves one of the author's results (see \cite{Sp16}). Note that $\tau_{\mathfrak{m}} = 1
$.
\begin{theorem} \label{mat-17}
	Let $\mathfrak{p}\not= \mathfrak{m}$ denote a prime ideal of a local Noetherian ring $(R,\mathfrak{m})$.
	Then $\tau_{\mathfrak{p}} =1$ if and only if $\dim R/\mathfrak{p} = 1$ and $R/\mathfrak{p}$ 
	is complete. 
\end{theorem}

\begin{proof}
	First remark that we may pass to the domain $S= R/\mathfrak{p}$, see the beginning of this 
	section. Then let $S$ denote a one-dimensional complete domain with $Q$ its quotient field. Since $Q/S$ 
	is torsion there is the short exact sequence $0 \to S \to Q \to H^1_{\mathfrak{m}}(S) \to 0$. Since $S$ is 
	complete and $H^1_{\mathfrak{m}}(S)$ is Artinian it follows that $Q$ is Matlis reflexiv. That is 
	$\dim_Q \Hom_S(Q,E_S)= 1$. 
	
	For the converse let $\dim_Q \Hom_S(Q,E_S) =1$. Then  $Q$ is Matlis reflexive. 
	The embedding $S \subset  Q$ implies that $S$ and $Q/S$ are Matlis reflexive. In particular, 
	$S$ is complete. Moreover $Q/S$ is Matlis reflexive as well. Therefore there 
	is a finitely generated submodule $N \subset Q/S$ such that the factor module is Artinian 
	(see \cite{Ee} or \cite{Zt}). Suppose that $\dim S > 1$. Then there is a prime ideal $\mathfrak{q} 
	\in \Supp_S Q/S$ with $\height \mathfrak{q} \geq 1$. Localizing at $\mathfrak{q}$ shows that 
	$(Q/S)_{\mathfrak{q}}$ is a finitely generated $S_{\mathfrak{q}}$-module. Because of $Q = Q(S_{\mathfrak{q}})$ 
	it follows that $Q(S_{\mathfrak{q}})$ is a finitely generated $S_{\mathfrak{q}}$-module. This is absurd 
	because of $\dim S_{\mathfrak{q}} \geq 1$, whence a contradiction to $\dim S > 1$.  That $\dim S =1$ follows also by 
	a different argument used by Daily and Marley (see \cite{DM} in the proof of their Proposition 3.4).
\end{proof}

We conclude with an example suggested by L. Winther Christensen et al. (see \cite{WFT}). 

\begin{remark} \label{ex-2} (see \cite[Remark 5.3]{WFT})
	Let $(R,\mathfrak{m})$ denote a local ring with $E_R = E_R(R/\mathfrak{m})$ its injective hull and $\dim R > 0$.  Let $n$ 
	denote a positive integer.  Because  $R/\mathfrak{m}^n$ is an $R$-module of finite length  
	there is an injection $R/\mathfrak{m}^n \to E_R^{\mu_n}$ with $\mu_n = \dim_\Bbbk \Hom_R(\Bbbk, R/\mathfrak{m}^n) 
	<\infty$ and  a commutative  diagram with exact rows 
		\[
		\xymatrix{0 \ar[r]& R/\mathfrak{m}^{n+1} \ar[r] \ar[d] &E_R^{\mu_{n+1}} \ar[d]^{\alpha_n} \\
			0  \ar[r] & R/\mathfrak{m}^n \ar[r] & E_R^{\mu_n} 
			}
		\]
		since $E_R^{\mu_n}$ is injective. We get an injection of inverse systems 
		$0 \to \{R/\mathfrak{m}^n\}_{n \geq 1} \to \{E_R^{\mu_n}\}_{n \geq 1} $. 
	Then the limit provides  an injection $0 \to \hat{R} \to E_R[|T|]$. 
	This is similar to the construction 
	in \cite[Remark 5.3]{WFT} 
	where an injection $0 \to R \to E_R[|T|]$ is shown. Because of $R \subseteq \hat{R}$ 
	we get the same with the additional information that it factors through $\hat{R}$. 
\end{remark}

In the following we shall add a few remarkable properties of the injective $R$-module $E_R[|T|]$. 
A question could be for which prime ideals $\mathfrak{p}$ 
the injective hull $E_R(R/\mathfrak{p})$ occurs with which multiplicity. 

\begin{example} \label{ex-1} (A)
	With the notation of \ref{ex-2}   we claim  the following 
	\[
	\mathfrak{p} \in \Ass_R  E[|T|] \, \text{ for any } \mathfrak{p} \in \Spec R \mbox{ and }  
		\mathfrak{P} \in \Ass_{\hat{R}}  E[|T|] \, \text{ for any } \mathfrak{P} \in \Spec \hat{R}
	\]
	Because of $ 0 \not= E_{R/\mathfrak{p}} \cong \Hom_R( R/\mathfrak{p},E_R)$ it follows 
	$ 0 \not= E_{R/\mathfrak{p}}[|T|] \cong \Hom_R(R/\mathfrak{p}, E_R[|T|])$
	and $\mathfrak{p} \in \Ass_R  E_R[|T|].$ Moreover, $E \cong E_{\hat{R}}$ and the second claim follows 
	as before (see also \cite[17.1.10]{WFH} for the author's argument). 
	 \\
	(B) Because $E[|T|]$ is an injective $R$-module $E[|T|] \cong 
	\oplus_{\mathfrak{p} \in \Spec R} E_R(R/\mathfrak{p})^{(Y_{\mathfrak{p}})}$ (see \cite{Me} or \cite{EJ}) with 
	$ \mu_{\mathfrak{p}} := \dim_{k(\mathfrak{p})} \Hom_{R/\mathfrak{p}}(k(\mathfrak{p}),(E[|T|])_{\mathfrak{p}})$ and $\operatorname{card} Y_{\mathfrak{p}} =
	  \mu_{\mathfrak{p}} $ (see also \ref{mat-12} (A) for the notation). 
	Let $\mathfrak{p}$ denote a prime ideal of $R$.
	Then  $\mathfrak{p} \in \Ass_R E[|T|]$ and $E_R(R/\mathfrak{p})$ is a direct summand of $E[|T|]$ (see 
	\cite[Lemma 1]{MeS}). That is $ \mu_{\mathfrak{p}}   \not=  0$ for all $\mathfrak{p}$. These numbers seem to be "highly" non finite, e.g. $\operatorname{card} Y_{\mathfrak{m}}=  \mu_{\mathfrak{m}}  = \dim_{\Bbbk} \Bbbk[|T|]$. \\
	(C)
	Next we investigate the Matlis dual $D_R( E[|T|])$. With the presentation above it yields that 
	\[
	\Hom_R(E[|T|], E) \cong \prod_{\mathfrak{p}} \Hom_R(E_R(R/\mathfrak{p}),E)^{Y_{\mathfrak{p}}}  
	\cong  \prod_{\mathfrak{p}}\big( \widehat{R_{\mathfrak{p}}^{(X_{\mathfrak{p}})}}\big)^{Y_{\mathfrak{p}}}.
	\]
	To this end recall the expression of $\Hom_R(E_R(R/\mathfrak{p}),E)$ by Enochs (see also the beginning 
	of Section 4). For the endomorphism ring of $E[|T|]$ it yields 
	\[
	\Hom_R(E[|T|], E[|T|]) 
	\cong ( \prod_{\mathfrak{p}}\big( \widehat{R_{\mathfrak{p}}^{(X_{\mathfrak{p}})}}\big)^{Y_{\mathfrak{p}}})^{\mathbb{N}}.
	\]
	 For completeness let us consider the injective $R$-module $E[T]$ as a submodule of $E[|T|]$ and therefore as a direct summand. It is easy to see 
	that $D_R(E[T]) = \Hom_R(E[T], E) \cong \hat{R}[|T|]$. For the endomorphism ring it follows that 
	$\Hom_R(E[T],E[T]) \cong \Hom_R(E,E[T])[|T|]$. Now  $\Hom_R(E,E[T]) \cong  \widehat{R^{(\mathbb{N})}}$
	(see  \cite[3.4.1]{EJ}) and therefore 
	$\Hom_R(E[T],E[T]) \cong (\widehat{R^{(\mathbb{N})}})^\mathbb{N}$.
\end{example}

While the endomorphism ring of $E$ is the completion it follows that the 
endomorphism rings of the injective modules $E[T]$ and $E[|T|]$ are rather "big" flat modules. 


\bibliographystyle{siam}

\bibliography{hart-1}

\end{document}